\documentclass[letterpaper, 10 pt, conference]{ieeeconf}

\usepackage{cite}
\usepackage{graphicx}
\usepackage{amsfonts}
\usepackage{amsmath} 
\usepackage{amssymb}

\usepackage{amsthm}	
\usepackage{dsfont}
\usepackage{enumerate}
\usepackage{xcolor}

\usepackage{booktabs}
\usepackage{subcaption}

\usepackage{geometry}
\newtheorem{thm}{Theorem}

\newtheorem{defn}{Definition}

\newtheorem{ass}{Assumption}

\newcommand{\argmin}{\mathop{\rm argmin}}

\newcommand{\R}{{\mathbb R}}

\newcommand{\trace}[1]{\mathop{\rm Tr}\left(#1\right)}

\renewcommand{\arraystretch}{0.9}

\newcommand{\Rset}{\mathbb{R}}
\newcommand{\Sset}{\mathbb{S}}

\newcommand{\Kcal}{{\mathcal{K}}}

\newcommand{\Mcal}{{\cal M}}
\newcommand{\Ncal}{{\mathcal{N}}}
\newcommand{\Pcal}{{\cal P}}

\newcounter{l1}
\newcounter{l2}
\newcounter{l3}
\newcommand{\bdotlist}{\begin{list}{$\bullet$}{}}
\newcommand{\bboxlist}{\begin{list}{$\Box$}{}}
\newcommand{\bbboxlist}{\begin{list}{\raisebox{.005in}{{\tiny
$\blacksquare$ \ \ }}}{}}
\newcommand{\bdashlist}{\begin{list}{$-$}{} }
\newcommand{\blist}{\begin{list}{}{} }
\newcommand{\barablist}{\begin{list}{\arabic{l1}}{\usecounter{l1}}}
\newcommand{\balphlist}{\begin{list}{(\alph{l2})}{\usecounter{l2}}}
\newcommand{\bAlphlist}{\begin{list}{\Alph{l2}.}{\usecounter{l2}}}
\newcommand{\bdiamlist}{\begin{list}{$\diamond$}{}}
\newcommand{\bromalist}{\begin{list}{(\roman{l3})}{\usecounter{l3}}}

\IEEEoverridecommandlockouts

\begin{document}

\title{\LARGE \bf A Distributionally Robust Approach to Regret Optimal Control using the Wasserstein Distance}
\author{Feras Al Taha,  \, Shuhao Yan, \, Eilyan Bitar 
\thanks{This paper will appear in the proceedings of the 2023 IEEE Conference on Decision and Control (CDC). The original version of this paper was posted on arXiv on April 23, 2023.}
\thanks{This work was supported in part by The Nature Conservancy, in part by the Cornell Atkinson Center for Sustainability, and in part by the Natural Sciences and Engineering Research Council of Canada.}
\thanks{The authors are with the School of Electrical and Computer Engineering, Cornell University, Ithaca, NY, 14853, USA.  Emails:  Feras Al Taha (foa6@cornell.edu), Shuhao Yan (sy499@cornell.edu),  Eilyan Bitar (eyb5@cornell.edu).  }%
}

\maketitle 

\begin{abstract} 
    This paper proposes a distributionally robust approach to regret optimal control of discrete-time linear dynamical systems with quadratic costs subject to a stochastic additive disturbance on the state process. The underlying probability distribution of the disturbance process is unknown, but assumed to lie in a given ball  of distributions defined in terms of the type-2 Wasserstein distance. In this framework, strictly causal linear disturbance feedback controllers are designed to minimize the worst-case expected regret. The regret incurred by a controller is defined as the difference between the cost it  incurs in response to a realization of the disturbance process and the cost incurred by the optimal noncausal controller which has perfect knowledge of the disturbance process realization at the outset. Building on a well-established duality theory for optimal transport problems, we derive a reformulation of the minimax regret optimal control problem as a tractable semidefinite program. Using the equivalent dual reformulation, we characterize a worst-case distribution achieving the worst-case expected regret in relation to the distribution at the center of the Wasserstein ball. We compare the minimax regret optimal control design method with the distributionally robust optimal control approach using an illustrative example and numerical experiments.
\end{abstract}

\section{Introduction} \label{sec:introduction}

There is growing interest in utilizing regret-based performance metrics, originally proposed in \cite{savage1951theory},  to design controllers for uncertain dynamical systems. Loosely speaking, regret measures the loss in performance incurred by a causal controller relative to a clairvoyant  controller that has complete knowledge of the underlying system dynamics and disturbance process at the outset. The regret minimization paradigm has been used to address a variety of estimation  and decision problems,  including robust  estimation \cite{eldar2004competitive, eldar2004linear}, multi-armed bandits \cite{bubeck2012regret},  online convex optimization \cite{shalev2012online}, and adaptive control \cite{abbasi2011regret}. 

In the context of linear quadratic (LQ) control problems with unknown state and input matrices, \cite{dean2018regret} and \cite{ouyang2019posterior} design learning-enabled controllers that achieve expected regret rates which scale optimally with the control horizon. 
More recently, \cite{9483023} and \cite{10061197} introduced the framework of regret optimal control for LQ problems, where the disturbance acting on the system is assumed to be adversarial in nature. Using operator-theoretic techniques from robust control, controllers in state-space representation are synthesized to minimize the worst-case regret over all possible disturbances with bounded energy. 
Building on this framework,  \cite{didier2022system} and \cite{martin2022safe} synthesize regret optimal controllers subject to state and input constraints. Using the system level parameterization \cite{anderson2019system}, they provide  equivalent reformulations of the resulting minimax regret optimal control problems as semidefinite programs (SDPs).
While the resulting controllers are robust with respect to the worst-case disturbance realization, they may be conservative if the underlying disturbance is actually stochastic in nature. 

Taking a probabilistic view, we propose an alternative approach to regret optimal control, where controllers are designed to minimize the worst-case expected regret over all probability distributions in a given ambiguity set specified in terms of the type-2 Wasserstein distance. In this setting, the expected regret is defined as the difference between an expected cost incurred by a causal controller and the expected cost incurred by the optimal noncausal controller with perfect knowledge of the disturbance trajectory at the outset. 
The proposed approach can also be thought of as an alternative to distributionally robust optimal (DRO) control, see, i.e., \cite{van2015distributionally,pmlr-v120-coppens20a,kim2023distributional} and references therein. 
A potential drawback of the DRO control approach, which aims to minimize the worst-case expected cost across a given ambiguity set of distributions, is that resulting controllers may incur a large expected regret for certain distributions in the ambiguity set. 
In contrast,  by minimizing the worst-case expected regret,  the approach proposed in this paper ensures that resulting controllers have uniformly small regret for all distributions in the given ambiguity set.
In Section \ref{sec:example}, we provide a motivating example that clearly illustrates some of the trade-offs inherent to these two methods.

\paragraph*{Summary of contributions} In this paper, we propose a new framework for regret optimal controller synthesis using distributionally robust optimization methods. 
The control problem is formulated as a minimax optimization problem, involving a worst-case expectation problem over a Wasserstein ball of distributions. 
Utilizing duality for distributionally robust optimization problems \cite{gao2022distributionally}, we obtain an equivalent convex reformulation of the worst-case expectation problem, which  facilitates the reformulation of the minimax regret optimal control problem as a SDP. 
Our result extends a number of related results on the exact reformulation of worst-case expectation problems involving quadratic objective functions \cite{nguyen2022distributionally,kuhn2019wasserstein}. We also provide an explicit characterization of a worst-case distribution achieving the worst-case expectation in relation to the central distribution in the ambiguity set. 

The remainder of this paper is organized as follows. In Section \ref{sec:formulation}, we formulate the minimax regret optimal control problem for linear disturbance feedback control policies.
In Section \ref{sec:example}, we provide a stylized example that clearly illustrates some of the strengths and limitations  of our proposed approach.
In Section \ref{sec:reformulation}, we derive an equivalent reformulation of the minimax regret optimal control problem as a SDP. 
In Section \ref{sec:experiments}, we present numerical experiments which show that the proposed approach can significantly outperform the distributionally robust optimal control approach in certain settings. We conclude the paper in Section \ref{sec:conclusion}. 

\subsubsection*{Notation} Let $\Rset$ and $\Rset_+$ denote the set of real numbers and nonnegative real numbers, respectively.  Let $\Sset^n$ denote the set of all symmetric matrices in $\Rset^{n \times n}$. Denote the cone of $n \times n$  real symmetric positive definite (semidefinite) matrices by $\Sset_{++}^{n}$ ($\Sset_+^{n}$). Denote the Euclidean norm of a vector $x \in \Rset^n$ by $\|x\|$. Given matrices $A, B \in \Sset^n$, the relation $A \succ B$ means $A - B \in \Sset^n_{++}$, and the relation $A \succeq B$ means $A-B \in \Sset^n_+$. Denote the maximum and minimum eigenvalues of a matrix $A \in \Rset^{n \times n}$ by $\lambda_{\rm max}(A)$ and $\lambda_{\rm min}(A)$, respectively. Given a matrix $A \in \Sset_+^n$,  the matrix $A^{\frac{1}{2}}$ denotes the (symmetric) positive semidefinite square root of $A$. Let $\Mcal(\Rset^{n})$ be the collection of Borel probability measures on $\Rset^{n}$ with finite second moments.

\section{Problem Formulation} \label{sec:formulation}

\subsection{System Model}
We consider discrete-time, linear time-varying systems evolving over a finite horizon $t=0, \dots, T-1$, with dynamics 
\begin{align} \label{eq:LTI}
x_{t+1} = A_tx_t + B_tu_t + w_t,
\end{align}
where $x_t \in \Rset^n$ is the \emph{system state}, $u_t \in \Rset^m$ is the \emph{control input}, and $w_t \in \Rset^n$ is the \emph{disturbance} acting on the system at  time $t$. We assume perfect state feedback, and that the system matrices $A_t \in \Rset^{n \times n}$ and input matrices $B_t \in \Rset^{n \times m} $ for $t=0, \dots, T-1$ are known at the outset.  Both the initial system state $x_0$ and  the disturbances $w_0, \dots, w_{T-1}$ are assumed to be random variables whose joint distribution $P$ is unknown but assumed to lie in a given compact set $\mathcal{P}$, termed the \emph{ambiguity set}.  

We denote the \emph{state, input,} and \emph{disturbance trajectories} by
\begin{align*}
x & :=   (x_0, \dots, x_T) \in \Rset^{N_x},\\
u &  := (u_0, \dots, u_{T-1}) \in \Rset^{N_u} ,\\
w &  := (x_0, w_0, \dots, w_{T-1}) \in \Rset^{N_x},
\end{align*}
where $N_x := n(T+1)$ and $N_u := mT$. 
To simplify notation, we have included the initial state $x_0$ as the first term in the disturbance trajectory $w$.  With these definitions, the dynamics \eqref{eq:LTI} can be expressed in terms of the following causal linear mapping from the  input  and disturbance trajectories to the  state trajectory:
\begin{align} \label{eq:state_traj}
x = Fu + Gw,
\end{align}
 where $F \in \Rset^{N_x \times N_u}$ and $G \in \Rset^{N_x \times N_x}$ are  block lower-triangular matrices. These matrices are straightforward to construct from the given state and input matrices $(A_t, \, B_t)$, where $t=0, \dots, T-1$. 

The cost incurred by an input trajectory $u$ and disturbance trajectory $w$ is defined as \begin{align} \label{eq:cost}
J(u, \, w) :=  x^\top Q x \, + \, u^\top R u, 
\end{align}
where $Q \in \Sset_+^{N_x}$ and $R \in \Sset_{++}^{N_u}$.

\subsection{Linear Disturbance Feedback Controllers}
 In this paper, we consider the design of  strictly causal  linear disturbance feedback controllers\footnote{It is straightforward to extend the results of this paper to allow for the optimization over \emph{affine} disturbance feedback policies. We forgo this more general treatment to streamline the presentation in this paper.}  of the form
\begin{align} \label{eq:DFcontroller}
u_t =  \sum_{k=0}^{t} K_{t,k} w_{k-1},  \quad \forall \,  t=0, \dots, T-1.
\end{align}
Here, we have used the notational convention $w_{-1} := x_0$. The  controller \eqref{eq:DFcontroller} can be  expressed as a linear mapping from the disturbance trajectory to the input trajectory given by
\begin{align}
u = Kw, \quad K \in \Kcal,
\end{align}
where $\Kcal  \subseteq \Rset^{N_u \times N_x}$   denotes the space of all block lower triangular matrices that  correspond to the disturbance feedback parameterization specified  in \eqref{eq:DFcontroller}.   

It is important to note that the  family of strictly causal linear disturbance feedback controllers is equivalent to the family of  causal linear state feedback controllers \cite{goulart2006optimization}.
In particular, given a  strictly causal disturbance feedback controller $K \in \Kcal$, the  change of variables\footnote{This invertible nonlinear transformation is related to the well-known Youla parameterization in linear systems \cite{youla1976modern}. The matrix $G$ is invertible since it is a lower triangular matrix with ones along its diagonal, and the matrix $I + KG^{-1}F$ is invertible since $KG^{-1}F$ is a  block strictly lower triangular matrix.}  $L := (I + KG^{-1}F)^{-1} K G^{-1}$ induces an equivalent causal state feedback controller $L$  that satisfies $Lx = Kw$ for all $w \in \Rset^{N_x}$ \cite{skaf2010design}.
However, from an optimization standpoint, the  disturbance feedback parameterization is preferred over the state feedback parameterization, because the cost function $J(Kw,\, w) $  resulting from the disturbance feedback parameterization is guaranteed to be convex with respect to the controller  $K$ \cite{goulart2006optimization}. 

\subsection{Minimax Regret Optimal Control}
We seek controllers belonging to the family $\Kcal$ that minimize the worst-case expected regret across all probability distributions in the given ambiguity set $\Pcal$.  The \emph{regret} incurred by a causal controller is defined as the difference between the cost it incurs and the cost incurred by the optimal noncausal controller (which knows the full disturbance trajectory at the outset). Formally, the regret incurred by a causal controller $K \in \Kcal$ in response to a disturbance $w$ is defined as
\begin{align} \label{eq:regret}
R(K, \, w) : = J(Kw, \, w)  \, -  \, J(u^\star(w), \, w),
\end{align}
where  $u^\star \hspace{-.05in} : \hspace{-.05in} \Rset^{N_x} \rightarrow \Rset^{N_u}$ denotes the  \emph{optimal noncausal controller}, which is defined as
\begin{align} \label{eq:opt_oracle}
u^\star(w) := \argmin_{u \in \Rset^{N_u}} J(u, \, w).
\end{align}

With these definitions in hand, the  \emph{minimax regret optimal control} (MROC)  problem can be formulated as
\begin{align} \label{eq:MRO_original}
\inf_{ K \in \Kcal} \sup_{P \in  \Pcal}   \mathds{E}_{P}\left[ R(K, \, w) \right].
\end{align}
To solve the MROC problem, an explicit characterization of the optimal noncausal controller cost  is needed. It is well known that the optimal noncausal controller exists, is unique, and  is linear in the disturbance trajectory $w$ (among all possible nonlinear controllers) \cite{hassibi1999indefinite, martin2022safe}. To see why this is true, note that $J(u, \, w)$ is a strictly convex quadratic function of $u$. It follows that the unique optimal noncausal controller can be obtained by solving the  optimality condition $\nabla_u J(u, \, w) = 0$, which yields
\begin{align*}
    u^\star(w) = K^\star w,
\end{align*}
where $K^\star :=   -(R + F^\top Q F)^{-1}F^\top Q G$.  
Using the above identity, the optimal noncausal controller cost can be expressed as
\begin{align*}
    J(K^\star w, \, w) = w^\top G^\top ( Q -   QFD^{-1}F^\top Q) G w,
\end{align*}
where $D :=  R + F^\top Q F$ is a positive definite matrix. Furthermore, by substituting \eqref{eq:state_traj} into \eqref{eq:cost}, and completing the square in \eqref{eq:cost} with respect to $u$,  the cost function $J(u, \, w)$  can be rewritten as
\begin{align} \label{eq:completing_square}
J(u, w)   = (u - K^\star w)^\top D (u - K^\star w) \, + \, J(K^\star w, \, w).
\end{align}
Notice that the first term  in \eqref{eq:completing_square} represents the regret incurred by the control trajectory $u$ in response to the disturbance $w$. Using the identity in \eqref{eq:completing_square}, the MROC problem can be recast as
\begin{align} \label{eq:MRO_reformulation}
\inf_{ K \in \Kcal} \sup_{P \in  \Pcal}   \mathds{E}_{P}\left[ w^\top(K - K^\star)^\top D (K - K^\star)w \right].
\end{align}
From \eqref{eq:MRO_reformulation}, it can be seen that the expected regret incurred by a controller $K \in \mathcal{K}$ under a distribution $P \in \mathcal{P}$ only depends on the distribution through its second moment.

\subsection{Distributionally Robust  Optimal Control}
In Sections \ref{sec:example} and \ref{sec:experiments}, we compare the \emph{minimax regret optimal} (MRO) controllers proposed in this paper with \emph{distributionally robust optimal} (DRO) controllers, which are defined as solutions to the following optimization problem:
\begin{align}  \label{eq:DRO_original}
\inf_{K \in \Kcal}   \sup_{P \in  \Pcal}   \mathds{E}_{P}\left[ J(Kw, \, w) \right].
\end{align}
DRO controllers minimize the worst-case expected cost across all distributions in a given ambiguity set $\Pcal$.

\section{Motivating Example} \label{sec:example}
Before delving into the main results of this paper, we first analyze a simple system that sheds light on the behavior of MRO  and DRO controllers and the nature of the adversarial distributions they hedge against. 

\subsection{System Description}
Consider a  one-dimensional ($n=1$) system that operates for a single time period ($T=1$):
\begin{align*}
    x_1 = x_0 + u_0 + w_0.
\end{align*}
In this setting, there is a single control input $u_0$, and the disturbance trajectory $ w= (x_0, w_0)$ is comprised of the initial state $x_0$ and initial disturbance $w_0$, which are assumed to have \emph{known} marginal distributions with zero mean and unit variance. The correlation coefficient between the initial state and disturbance is assumed to be \emph{unknown}. Given these assumptions, the ambiguity set of distributions can be expressed as
\begin{align*}
    \Pcal :=  \left\{P \in \Mcal(\Rset^2) \, \middle\vert  \begin{array}{l}  \mathds{E}_{P}[w] = 0,  \, \mathds{E}_{P}[ww^\top] = \begin{bmatrix}
        1 & \rho \\ \rho & 1
    \end{bmatrix}
      \\   \rho \in [-1,1] \, \text{ for } \, w \sim P   \end{array} \hspace{-.05in}  \right\}, 
\end{align*}
where $\rho$ denotes the correlation coefficient between $x_0$ and $w_0$.
We consider a cost function that penalizes the terminal state and control input according to
\begin{align*}
    J(u_0,w) \, =  \, x_1^2 \, + \, c u_0^2,
\end{align*}
where  $c \ge 1$ is a given parameter. 

{
\setlength{\tabcolsep}{6pt}
\newcommand{\ra}[1]{\renewcommand{\arraystretch}{#1}}
\begin{table}[t] \centering
    \ra{1.6}
    \begin{tabular}{lll} 
         \toprule
         \multicolumn{1}{l}{\bf Method}  & \multicolumn{1}{l}{\bf Control Policy} & \multicolumn{1}{l}{\bf Expected Cost}  \\ \midrule
         Optimal Noncausal   & $ -\left(\dfrac{1}{1+c}\right)\left(x_0+w_0\right)$ & $ J^\star$  \vspace{.145in} \\ 
         Optimal Causal  & $ -\left(\dfrac{1+\rho}{1+c}\right) x_0$ &  $J^\star   \, + \, \dfrac{1-\rho^2}{\hspace{-.06in} 1+c}$ \vspace{.145in}\\ 
         MRO & $ -\left(\dfrac{1}{1+c}\right) x_0$ & $ J^\star   \, + \, \dfrac{1}{1+c}$ \vspace{.145in}\\ 
         DRO & $ -\left(\dfrac{2}{1+c} \right) x_0$ & $J^\star   \, + \, 2 \left(\dfrac{1-\rho}{1+c} \right)$ \vspace{.01in}\\ 
         \bottomrule
    \end{tabular}
    \caption{Comparison of different control design methods applied to the specific system described in Section \ref{sec:example} for $c \geq 1$. The expected cost incurred by the optimal noncausal policy is given by   $J^\star :=  2c (1+\rho)/(1+c)$.} 
    \label{tab:Ks_costs}
\end{table}
}

We examine four different controllers: (i) the \emph{optimal noncausal controller}  defined in \eqref{eq:opt_oracle},  (ii) the \emph{MRO controller} defined in \eqref{eq:MRO_original}, (iii) the \emph{DRO controller} defined in \eqref{eq:DRO_original}, and  (iv) the \emph{optimal causal controller}, which minimizes the expected cost given complete knowledge of the underlying distribution $P\in\Pcal$. For each of the causal control design methods (ii)-(iv), we restrict our attention to linear policies of the form:
\begin{align*}
    u_0 \, = \, -kx_0, 
\end{align*}
where $k \geq 0$ denotes the control gain.  Closed-form expressions for the optimal controllers and their expected costs  are provided in Table \ref{tab:Ks_costs}. Due to space limitations, the derivation of these expressions is left as an exercise for the reader. 

\begin{figure}[ht]
    \centering
    \includegraphics[width=.9\columnwidth]{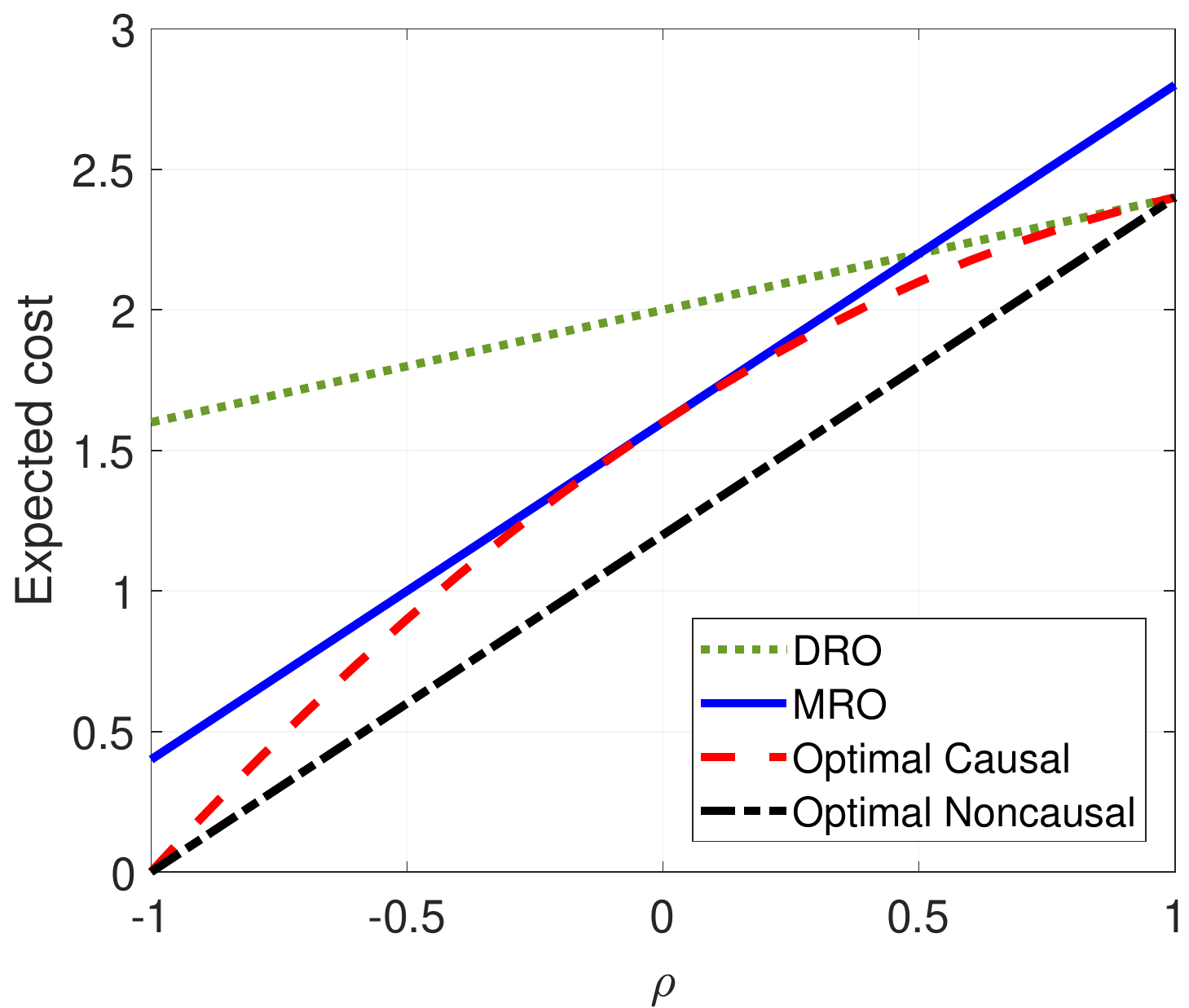}
    \caption{Expected cost incurred by different controllers as a function of the correlation coefficient $\rho$ for $c=3/2$.}
    \label{fig:costs}
\end{figure}

\subsection{Discussion}
In Fig. \ref{fig:costs}, we plot the expected cost of each controller as a function of the correlation coefficient $\rho$. The expected cost of the DRO controller is greatest at $\rho = 1$, i.e., when the initial state $x_0$ and the disturbance $w_0$ are perfectly positively correlated. This is unsurprising since heightened correlation between the initial state and the disturbance requires greater control effort to minimize the expected cost. 
Unlike the DRO controller, which minimizes the expected cost against the worst-case distribution in the ambiguity set, the MRO controller is found to strike a balance in compensating for both ``good'' $(\rho=-1)$ and ``bad'' ($\rho =1)$ distributions in the ambiguity set, i.e., distributions that result in lower and higher expected costs, respectively. In fact, the MRO controller is equivalent to the optimal causal controller corresponding to the correlation coefficient $\rho =0$. Thus, although the MRO controller incurs a greater expected cost than the DRO controller in the worst case, it strictly outperforms the DRO controller for all distributions in the ambiguity set with correlation coefficients $\rho \in [-1, \, 1/2)$.  
Moreover, the MRO controller achieves a constant expected regret of $1/(1+c)$ over the ambiguity set, while the DRO controller exhibits an expected regret of  $4/(1+c)$ for distributions with $\rho= -1$.

\section{Tractable Convex Reformulation} \label{sec:reformulation}

A primary challenge in solving the MROC problem \eqref{eq:MRO_reformulation} is that the inner maximization involves a supremum over an ambiguity set of distributions that may be nonconvex and infinite dimensional. To address this challenge, we consider ambiguity sets that are defined  as the set of all distributions that are \emph{close} to a nominal distribution with respect to the Wasserstein distance. The Wasserstein distance between two distributions can be interpreted as the minimum ``work'' or transportation cost required to move the probability mass from one distribution to the other. 
By defining the ambiguity set in this manner, we are able to draw on a well-established duality theory for optimal transport problems to reformulate the inner maximization in \eqref{eq:MRO_reformulation} as a tractable finite-dimensional convex optimization problem \cite{blanchet2019quantifying, gao2022distributionally}. As one of our main contributions in Theorem \ref{thm:sdp of MROC}, we  provide an exact reformulation of the MROC problem \eqref{eq:MRO_reformulation} as a SDP. In Theorem \ref{thm:worst_dist}, we also provide a structural characterization of a worst-case distribution achieving the supremum in the inner maximization in \eqref{eq:MRO_reformulation}. 

\subsection{Strong Duality for Worst-Case Expectation Problems}
Before presenting our main results, it is necessary to introduce some additional notation and define the  type-2 Wasserstein distance between probability measures. 

\begin{defn}[Type-2 Wasserstein Distance] \rm \label{defn:wass_dist} The \emph{type-2 Wasserstein distance} between two distributions $P_1, P_2 \in \Mcal(\Rset^{N_x})$  is defined as 
\begin{align*}
    W_2(P_1, \, P_2)^2 := \hspace{-.075in}\inf_{\pi \in \Pi(P_1, \, P_2)}  \int_{\Rset^{N_x} \times \Rset^{N_x}} \hspace{-.05in} \| z_1 - z_2 \|^2 \pi(d z_1, \, d z_2),
\end{align*}
 where $\Pi(P_1, \, P_2)$ denotes the set of all joint distributions  in $\Mcal(\Rset^{N_x} \times \Rset^{N_x})$ with marginal distributions $P_1$ and $P_2$.
\end{defn}
Given a \emph{nominal distribution} $P_0 \in \Mcal( \Rset^{N_x})$, we define the ambiguity set $\Pcal$ as the set of all distributions 
whose type-2 Wasserstein distance to $P_0$ is at most $r \geq 0$, i.e., 
\begin{align} \label{eq:amb_set}
    \Pcal := \{ P \in \Mcal( \Rset^{N_x}) \, | \, W_2(P, P_0) \leq r \}.
\end{align}
The \emph{radius of the ambiguity set}, $r$, reflects the degree of uncertainty surrounding the accuracy of the nominal distribution. For example, when the central distribution is estimated from independent and identically distributed (i.i.d.) data, the radius of the ambiguity set can be selected to ensure that the ambiguity set contains the data-generating distribution with a desired confidence level, e.g., using the techniques described in \cite{mohajerin2018data}. 
If the radius of the ambiguity set is decreased to zero, then the ambiguity set collapses to a singleton set that only contains the nominal distribution, and the MROC problem \eqref{eq:MRO_original}  reduces to a conventional (ambiguity-free) stochastic optimal control problem given by $\inf_{K \in \Kcal} \mathds{E}_{P_0}[J(Kw, \, w)]$. 

To solve the MROC problem \eqref{eq:MRO_reformulation}, one must contend with the inner maximization (worst-case expectation) over the type-2 Wasserstein ball of distributions. Recently, it has been shown that strong duality holds for such infinite-dimensional maximization problems for a large family of objective functions \cite{blanchet2019quantifying, gao2022distributionally, mohajerin2018data}. 
In particular, the following known result from the literature provides conditions under which strong duality is guaranteed to hold, along with a useful characterization of the dual problem as a one-dimensional convex minimization problem. 

\begin{thm}[Strong Duality, \cite{gao2022distributionally}] \label{thm:strong_duality} \rm Let $r>0$ and consider a family of worst-case expectation problems given by
\begin{align*}
  \sup_{P \in \Pcal} \mathds{E}_P[f(w)],  
\end{align*}
where $f: \Rset^{N_x} \rightarrow \Rset$ is a Borel-measurable function such that $\mathds{E}_{P_0}[|f(w)|] < \infty$. 
Then, strong duality holds, and
\begin{align} \label{eq:strong_duality}
   \sup_{P \in \Pcal} \mathds{E}_P[f(w)]   = \inf_{\gamma \geq 0} \left\{\gamma r^2  - \mathds{E}_{P_0}\left[  \phi(\gamma, \, w)\right]   \right\},
\end{align}
where the function $\phi: \Rset \times \R^{N_x} \rightarrow \Rset \cup \{-\infty\}$ is defined as
\begin{align*}
  \phi(\gamma, \, w)  := \inf_{z \in \Rset^{N_x}}  \{ \gamma \| z - w \|^2 - f(z) \}.
\end{align*}
Additionally, if $\inf\{\gamma \geq 0 \, | \,  \mathds{E}_{P_0}[\phi(\gamma, \, w)]  > - \infty \} < \infty$, then the primal and dual optimal values in \eqref{eq:strong_duality} are finite.
\end{thm}

Building on Theorem \ref{thm:strong_duality}, we provide conditions which enable the reformulation of worst-case expectation problems involving general quadratic objective functions as SDPs. We rely on the following assumption. 

\begin{ass}[Nominal Distribution] \label{ass:AC} \rm The nominal distribution $P_0$ is absolutely continuous with respect to the Lebesgue measure on $\Rset^{N_x}$.
\end{ass}

\begin{thm}[Strong Duality for Quadratic Objectives]\label{thm:strong_duality_quadratic} \rm Let $r >0$ and Assumption \ref{ass:AC} hold. Consider a family of worst-case expectation problems given by
\begin{align} \label{eq:quadratic}
  \sup_{P \in \Pcal} \mathds{E}_P[w^\top C w],  
\end{align}
    where the matrix $C \in \Sset^{N_x}$ is assumed to satisfy $\lambda_{\rm max}(C) \neq 0$.  Then, the optimal value of \eqref{eq:quadratic} is finite and equal to the optimal value of the following convex program:
    \begin{align} \label{eq:quad_SDP_reform}
    \inf_{\gamma \geq 0} \{ \gamma(r^2 - \trace{M_0}) + \gamma^2 \trace{M_0(\gamma I - C)^{-1}}    \, | \,  \gamma I \succ C\}, 
    \end{align}
    where $M_0 := \mathds{E}_{P_0}[w w ^\top]$.
\end{thm}

\begin{proof} For any matrix $C \in \Sset^{N_x}$, the  function $z \mapsto z^\top C z$ is Borel measurable, and satisfies $\mathds{E}_{P_0}[|w^\top C w|] < \infty $ since the nominal distribution $P_0$ has a finite second moment by assumption. Then, by Theorem \ref{thm:strong_duality}, it holds that
\begin{align} \label{eq:quad_primal_dual}
     \sup_{P \in \Pcal} \mathds{E}_P[w^\top C w] \, = \,  \inf_{\gamma \geq 0} \left\{\gamma r^2  - \mathds{E}_{P_0}\left[  \phi(\gamma, \, w)\right]   \right\},
\end{align}
where 
\begin{align*}
    \phi(\gamma, \,  w) = \hspace{-.075in} \inf_{z \in \Rset^{N_x}}  \left\{  z^\top (\gamma I - C) z  - 2 \gamma w^\top z \right\}  + \gamma \|w\|^2.
\end{align*}
Also, by Theorem \ref{thm:strong_duality}, the primal and dual optimal values  in \eqref{eq:quad_primal_dual} are finite if $\inf\{\gamma \geq 0 \, | \,  \mathds{E}_{P_0}[\phi(\gamma, \, w)]  > - \infty \} < \infty$. Under the conditions stated in Theorem \ref{thm:strong_duality_quadratic}, we shall prove that
 $ \mathds{E}_{P_0}[\phi(\gamma, \, w)]  > - \infty $ if and only if $\gamma I - C \succ 0$. The implication in the ``if'' direction  is clear. 

We prove the implication in the ``only if'' direction using contraposition, i.e., we shall prove that $\mathds{E}_{P_0}[\phi(\gamma, \, w)]  = - \infty $ if $\gamma I - C \nsucc 0 $. To streamline notation, let $C_\gamma := \gamma I - C $. Additionally, let $v \in \Rset^{N_x}$ denote an eigenvector of the matrix $C_\gamma$ associated with the smallest eigenvalue $\lambda_{\rm min}(C_\gamma)$.

Note that $C_\gamma \nsucc 0 $ is equivalent to  $\lambda_{\rm min}(C_\gamma ) \leq 0$. If $\lambda_{\rm min}(C_\gamma ) < 0$, then $\inf_{z \in \Rset^{N_x}} \left\{  z^\top C_\gamma z  - 2 \gamma w^\top z \right\} = -\infty$ for any $w \in \Rset^{N_x}$, which implies that $\mathds{E}_{P_0}[\phi(\gamma, \, w)]  = - \infty $. 
If $\lambda_{\rm min}(C_\gamma ) = 0$, then
\begingroup
\allowdisplaybreaks
\begin{align*}
    \inf_{z \in \Rset^{N_x}}  \left\{  z^\top C_\gamma  z  \right. &  \left. - 2 \gamma w^\top z  \right\}  \\ 
    & \leq \inf_{\theta \in \Rset} \left\{  (\theta v)^\top C_\gamma (\theta v)  - 2 \gamma w^\top (\theta v) \right\} \\
     & = \inf_{\theta \in \Rset} \left\{    -2 \gamma   w^\top (\theta v) \right\}\\
    & = \inf_{\theta \in \Rset} \left\{    -2 \lambda_{\rm max}(C) w^\top (\theta v) \right\}\\
    & = -\infty \ P_0\text{-almost surely.}
\end{align*}
\endgroup
The first equality follows from the assumption that $v$ is an eigenvector of $C_\gamma$ associated with the eigenvalue $\lambda_{\rm min}(C_\gamma) = 0$.  The second equality  follows from $\lambda_{\rm min}(C_\gamma) = 0$, which is equivalent to $\gamma = \lambda_{\rm max}(C)$. The final equality follows from the assumption that $\lambda_{\rm max}(C) \neq 0$ and Assumption \ref{ass:AC} (which implies that $w^\top v \neq 0$ $P_0$-almost surely).    This proves that   $ \mathds{E}_{P_0}[\phi(\gamma, \, w)]  > - \infty $ if and only if $C_\gamma  \succ 0$. 
It follows that 
\begin{align*}
\inf\{\gamma \geq 0 \, | \,  \mathds{E}_{P_0}[\phi(\gamma, \, w)]  > - \infty \} & = \inf\{\gamma \geq 0 \, | \, C_\gamma \succ 0 \}\\
& = \max\{ \lambda_{\rm max}(C), \, 0 \} \\
& < \infty,
\end{align*}
proving finiteness of the primal and dual optimal values in \eqref{eq:quad_primal_dual}. 

We complete the proof by showing that the right-hand side of \eqref{eq:quad_primal_dual} is equal to \eqref{eq:quad_SDP_reform}.  From  previous arguments, it follows that
\begin{multline}
\inf_{\gamma \geq 0} \left\{\gamma r^2  - \mathds{E}_{P_0}\left[  \phi(\gamma, \, w)\right]   \right\} \\
= 
\inf_{\gamma \geq 0} \left\{\gamma r^2  - \mathds{E}_{P_0}\left[  \phi(\gamma, \, w)\right]   |\, C_\gamma \succ 0\right\}, \nonumber
\end{multline}
which leads to \eqref{eq:quad_SDP_reform} as 
$\inf_{z \in \Rset^{N_x}}  \bigl\{  z^\top C_\gamma z  \!-\! 2 \gamma w^\top z \bigr\} = -\gamma^2 w^\top C_\gamma^{-1} w$ under the condition $C_\gamma \succ 0$. \end{proof}

Together with Assumption \ref{ass:AC}, the condition $\lambda_{\rm max}(C)\neq0$ ensures that the dual optimal value \eqref{eq:quad_SDP_reform} is finite. While this condition permits a variety of  quadratic objective functions, including definite and indefinite quadratic functions, it does exclude negative semidefinite quadratic functions with $\lambda_{\rm max}(C) = 0$. 

We also note that Theorem \ref{thm:strong_duality_quadratic} extends a number of related results in the literature \cite{nguyen2022distributionally, kuhn2019wasserstein} by expanding the family of nominal distributions under which the worst-case expectation problem \eqref{eq:quadratic} is guaranteed to admit an equivalent convex reformulation as \eqref{eq:quad_SDP_reform}. 
In particular, the  convex reformulation provided in \cite[Theorem 16]{kuhn2019wasserstein} relies on the assumption that the nominal distribution is a (possibly degenerate) elliptical distribution, while Assumption \ref{ass:AC} only requires that the nominal distribution be absolutely continuous with respect to the Lebesgue measure on $\Rset^{N_x}$. While the family of probability distributions satisfying Assumption \ref{ass:AC} is quite broad, it does rule out  nominal distributions supported on  lower dimensional manifolds (e.g.,  degenerate distributions), and distributions supported on finite sets. For a comprehensive treatment of  nominal distributions supported on finite sets, we refer the reader to \cite{mohajerin2018data}, which provides tractable convex reformulations for an array of worst-case expectation problems over type-1 Wasserstein balls centered at empirical distributions. 

\subsection{Worst-Case Distribution Characterization}

Using the dual reformulation of the quadratic worst-case expectation problem \eqref{eq:quadratic} provided in Theorem \ref{thm:strong_duality_quadratic}, 
we now characterize a worst-case distribution which attains the optimal value of the worst-case expectation problem \eqref{eq:quadratic}. 

\begin{thm}[Worst-Case Distribution] \rm \label{thm:worst_dist}
    Let the assumptions of Theorem \ref{thm:strong_duality_quadratic} hold and let $\gamma^\star\ge 0$ satisfy $\gamma^\star I \succ C$ and be a solution to the algebraic equation
    \begin{align} \label{eq:opt_cond}
         {\rm Tr} \big( \big( \gamma(\gamma I- C)^{-1} - I \big)^2 M_0 \big) = r^2.
    \end{align}
   Let $P^\star \in \Mcal(\Rset^{N_x})$ denote the distribution of 
    \begin{align} \label{eq:w_star}
        w^\star := \gamma^\star (\gamma^\star I - C)^{-1} w, 
    \end{align}
    where $w \sim P_0$. Then, the distribution $P^\star$ belongs to the ambiguity set $\Pcal$ and it attains the optimal value of the worst-case expectation problem \eqref{eq:quadratic}. 
\end{thm}
\begin{proof}
    To prove the desired result, we need to show that $P^\star$ is both feasible and optimal in \eqref{eq:quadratic}. To streamline notation, let $C_\gamma := \gamma I - C$.
    
    First, we show that the distribution $P^\star$ belongs to the ambiguity set $\Pcal$ by verifying that it satisfies $W_2(P^\star,P_0) \le r$.  
    To achieve this, let  $\pi\in\Pi(P^\star,P_0)$ be the joint distribution of $w^\star$ and $w$. Then, we have that
    \begingroup
    \allowdisplaybreaks
    \begin{align*}
        W_2( &P^\star, P_0)^2 \\
        &\le \mathds{E}_\pi \left [(w^\star-w)^\top(w^\star-w) \right]\\
        &= {\rm Tr} \big(  \mathds{E}_{P^\star}\big [w^\star {w^\star}^\top \big] -2\mathds{E}_{\pi}\big [w {w^\star}^\top \big] + \mathds{E}_{P_0}\big [w w^\top \big] \big ) \\
        &\overset{(a)}{=} {\rm Tr} \big(    {\gamma^\star}^2 C_{\gamma^\star}^{-1} M_0 C_{\gamma^\star}^{-1} -2 \gamma^\star C_{\gamma^\star}^{-1} M_0 + M_0 \big )\\
        &= {\rm Tr} \big(    \big (\gamma^\star C_{\gamma^\star}^{-1} - I \big )^2 M_0 \big )\\
        &\overset{(b)}{=} r^2,
    \end{align*}
    \endgroup
    where (a) holds by \eqref{eq:w_star} and (b) holds by \eqref{eq:opt_cond}. 
    Hence, $P^\star \in \Pcal$ and is feasible for \eqref{eq:quadratic}. 

    Next, we prove that the distribution $P^\star$ is optimal for the worst-case expectation problem \eqref{eq:quadratic}.
    Let $g:\Rset_+ \to \Rset$ denote the dual objective function in \eqref{eq:quad_SDP_reform}.
    It can be expressed as
    \begin{align*}
        g(\gamma) &= \gamma (r^2 - \trace{M_0}) + \gamma^2 \trace{M_0 C_\gamma^{-1}}. 
    \end{align*} 
    The function $g$ is strictly convex on the interval $(\lambda_{\rm max}(C),\, \infty)$ since its second derivative satisfies
    \begin{align*}
        \frac{d^2 g(\gamma) }{d\gamma ^2 } = 2 \trace{M_0 C^2 C_\gamma^{-3}} > 0
    \end{align*}
     for all $\gamma \in (\lambda_{\rm max}(C) ,\,  \infty)$. 
     Also, note that $g$ goes to infinity as $\gamma$ tends to infinity or approaches $\lambda_{\rm max}(C)$ from the right.
    Consequently, the dual function $g$ has a unique minimizer $\gamma^\star$ in the interval $(\lambda_{\rm max}(C) ,\,  \infty)$.
    Hence, the first order optimality condition is a necessary and sufficient condition for optimality on the interval $(\lambda_{\rm max}(C),\, \infty)$. This condition is given by
    \begin{align*}
        \frac{dg(\gamma) }{d\gamma} &= r^2 - {\rm Tr} \big( M_0 - 2 \gamma C_\gamma^{-1} M_0 +{\gamma}^2 C_\gamma^{-2}M_0 \big ) = 0, 
    \end{align*}
    which can be rewritten as \eqref{eq:opt_cond}.
    
    By Theorem \ref{thm:strong_duality_quadratic}, the optimal value of the primal problem \eqref{eq:quadratic} is finite and equal to the optimal value of the dual problem \eqref{eq:quad_SDP_reform}. 
    Hence, to prove that $P^\star$ attains the supremum of \eqref{eq:quadratic}, it suffices to show that $\mathbb{E}_{P^\star}[{w^\star}^\top C w^\star] = g(\gamma^\star)$. It holds that
    \begingroup
    \allowdisplaybreaks
    \begin{align*}
        g(\gamma^\star)
        &\overset{(a)}{=} \gamma^\star {\rm Tr} \big( \big (\gamma^\star C_{\gamma^\star}^{-1} - I\big )^2 M_0 - M_0 + \gamma^\star C_{\gamma^\star}^{-1} M_0\big )\\
        &= \gamma^\star {\rm Tr} \big( \big (-\gamma^\star C_{\gamma^\star}^{-1} +{\gamma^\star}^2 C_{\gamma^\star}^{-2} \big )M_0\big )\\
        &= {\gamma^\star}^2 {\rm Tr} \big( C^{-1}_{\gamma^\star} M_0 \big (-I + \gamma^\star C_{\gamma^\star}^{-1}\big )\big )\\
        &\overset{(b)}{=}  {\rm Tr} \big(  {\gamma^\star}^2 C^{-1}_{\gamma^\star} M_0 C_{\gamma^\star}^{-1} C \big )\\
        &\overset{(c)}{=} \mathds{E}_{P^\star}[{w^\star}^\top C w^\star],
    \end{align*}
    \endgroup
    where (a) follows from plugging \eqref{eq:opt_cond} into $g(\gamma^\star)$;
    (b) holds since $-I + \gamma^\star C_{\gamma^\star}^{-1} = -C_{\gamma^\star}^{-1}C_{\gamma^\star} + \gamma^\star C_{\gamma^\star}^{-1} = C_{\gamma^\star}^{-1} C$;
    and (c) follows from \eqref{eq:w_star}. 
    Therefore, $P^\star$ is optimal and the desired result holds. \end{proof}

In Theorem \ref{thm:worst_dist}, the proposed worst-case distribution $P^\star$ is given by the push-forward measure of $P_0$ through the linear transformation $z \mapsto \gamma^\star (\gamma^\star I - C)^{-1} z$. 
It is also possible to show that $P^\star$ is an extremal distribution in the ambiguity set $\Pcal$.
In particular, by using a well-known lower bound on the type-2 Wasserstein distance due to Gelbrich \cite[Theorem 2.1]{gelbrich1990formula},  one can show that $W_2(P^\star,P_0) \ge r$.

\subsection{Semidefinite Programming Reformulation}
Building on Theorem \ref{thm:strong_duality_quadratic}, we now show how to reformulate the MROC problem \eqref{eq:MRO_original} as a SDP in Theorem \ref{thm:sdp of MROC}.

\begin{thm}[SDP Reformulation of MROC Problem]\label{thm:sdp of MROC} \rm Let $r>0$ and suppose that $K^\star\notin \Kcal$ and Assumption \ref{ass:AC}
holds. Then, the MROC problem \eqref{eq:MRO_original} can be equivalently reformulated as the following SDP:
\begin{subequations}\label{eq:sdp for MROC}
\begin{align}
    \inf~  &\gamma (r^2 - \trace{M_0}) + \trace{X} \\
    \text{ s.t. }
   & \ K\in\Kcal,\, \gamma\geq0,\, X\in\Sset_{+}^{N_x}, 
   \nonumber \\
    & \hspace{0.25mm}\begin{bmatrix}
        \gamma I & (K-K^\star)^\top\\
        K-K^\star & D^{-1}
    \end{bmatrix}\succ 0, \label{eq:lmi 1}\\
   & \begin{bmatrix}
        X & \gamma M_0^{\frac{1}{2}} & 0_{N_u\times N_x}^\top\\
        \gamma M_0^{\frac{1}{2}} & \gamma I & (K-K^\star)^\top \\
        0_{N_u\times N_x} & K-K^\star & D^{-1}
    \end{bmatrix}\succeq 0, \label{eq:lmi 2}
\end{align}
\end{subequations}
where $0_{N_u \times N_x}$ denotes a $N_u$-by-$N_x$ matrix of all zeros and the decision variables are $K$, $\gamma$ and $X$. 
\end{thm}
\begin{proof}
To streamline notation,   define the matrix $$C_K:=(K-K^\star)^\top D (K-K^\star).$$
The condition $K^\star\notin \Kcal$ ensures that $\lambda_{\rm max}(C_K)>0$ for all $K \in \mathcal{K}$. 
This, together with Assumption \ref{ass:AC}, allows us to apply Theorem \ref{thm:strong_duality_quadratic} to reformulate the MROC problem as
\begin{align}
   & \hspace{-4mm}\inf_{K\in\Kcal}\sup_{P\in\Pcal}\mathds{E}_P[ w^\top C_K  w ] \nonumber\\
    &\hspace{-4mm}=\!
    \inf_{K\in\Kcal,\gamma\geq 0}\bigl\{
    \gamma (r^2 - \trace{M_0}) \nonumber\\
    & \hspace{14mm}+ \gamma^2 \trace{M_0 (\gamma I - C_K )^{-1}}
    \,|\, \gamma I \succ C_K \bigr\}. \label{eq:write inner max as min}
\end{align}
Since $D \succ 0$, it follows from  the Schur complement theorem that the constraint $\gamma I\succ C_K$ is equivalent to the linear matrix inequality \eqref{eq:lmi 1}. Now, observe that the second term $\gamma^2 \trace{M_0 (\gamma I - C_K )^{-1}}$ in the right-hand side of \eqref{eq:write inner max as min} can be expressed as
$$
\min_{X\succeq 0} \big\{ \trace{X} \, \big|
\,  X  \succeq   \gamma M_0^{\frac{1}{2}} (\gamma I - C_K )^{-1}  \gamma M_0^{\frac{1}{2}}
\big\}.
$$
Once again, using the Schur complement theorem, we can reformulate the constraint $X \!\succeq\! \gamma M_0^{\frac{1}{2}} (\gamma I - C_K )^{-1}  \gamma M_0^{\frac{1}{2}}$ as
\begin{equation}
    \begin{bmatrix}
        X & \gamma M_0^{\frac{1}{2}}\\
        \gamma M_0^{\frac{1}{2}} & \gamma I -C_K
    \end{bmatrix}\succeq 0, \label{eq:nonlinear matrix inequality}
\end{equation}
since $\gamma I-C_K \succ 0$. To prove equivalence between the matrix inequality \eqref{eq:nonlinear matrix inequality} and \eqref{eq:lmi 2},  it is helpful to rewrite \eqref{eq:nonlinear matrix inequality} as 
\begin{equation}
    \begin{bmatrix}
        X & \gamma M_0^{\frac{1}{2}}\\
        \gamma M_0^{\frac{1}{2}} & \gamma I 
    \end{bmatrix}
    -
    \begin{bmatrix}
        0_{N_u\times N_x}^\top \\ (K-K^\star)^\top
    \end{bmatrix}
    D
    \begin{bmatrix}
        0_{N_u\times N_x}^\top \\ (K-K^\star)^\top
    \end{bmatrix}^\top \succeq 0. \nonumber
\end{equation}
Since $D \succ 0$, it follows from  the Schur complement theorem that the above constraint is equivalent to \eqref{eq:lmi 2}. This proves that problem \eqref{eq:write inner max as min} is equivalent to the SDP \eqref{eq:sdp for MROC} and completes the proof.
\end{proof}
Note that the assumption $K^\star\notin\Kcal$ is without loss of generality. If the optimal noncausal controller $K^\star$ happens to be causal, i.e., $K^\star\in\Kcal$, then the minimax regret optimal controller is just $u=K^\star w$, which avoids the need to solve problem \eqref{eq:MRO_original}. 

The size of the SDP in \eqref{eq:sdp for MROC} grows polynomially with the state dimension $n$, input dimension $m$, and time horizon $T$.  Although SDPs can be solved with arbitrary accuracy in polynomial time using interior point methods, in practice they are expensive to solve in high dimensions. To improve scalability, one can exploit sparsity in the problem data \cite{vandenberghe2015chordal,zheng2020chordal}, or employ approximations of the positive semidefinite cone \cite{louca2016hierarchy,ahmadi2019dsos}.
\section{Numerical Experiments} \label{sec:experiments}

In this section, we compare the MRO and DRO control design methods within the context of data-driven control.
We consider a one-dimensional ($n=1$) controlled random walk with dynamics given by
\begin{align*}
    x_{t+1} = x_t + u_t + w_t
\end{align*}
 for $t=0,...,T-1$.  In our experiments, we consider a time horizon of $T=10$ and take the cost matrices to be $Q, \, R = I$.
The disturbance trajectory $w$ is assumed to have a  Gaussian distribution given by $\mathcal{N}(\mu,I)$. 

We assume that the controllers have access to $N$ i.i.d. samples $w^{(1)}, \dots, w^{(N)}$ of the disturbance trajectory,  but do not know the underlying distribution from which the data is generated. 
Using these data, we consider an ambiguity set $\mathcal{P}$ that is centered at a nominal distribution $P_0 \in \Mcal( \Rset^{N_x})$ whose second moment matrix is given by the empirical estimate:
\begin{align}  \label{eq:M_hat}
    M_0 = \frac{1}{N} \sum_{i=1}^N w^{(i)} {w^{(i)}}^\top.
\end{align}
Note that the SDP reformulations of the MROC and DROC problems only depend on the second moment matrix $M_0$ of the nominal distribution $P_0$, so specifying its second moment suffices.
In our experiments, we vary the radius $r$ of the ambiguity set  to investigate its effect on the performance of the resulting MRO and DRO controllers. 

\begin{figure}[t!]
    \centering
    \subfloat[$\mu=0$\label{fig:mu_zero}]{\includegraphics[width=0.9\columnwidth]{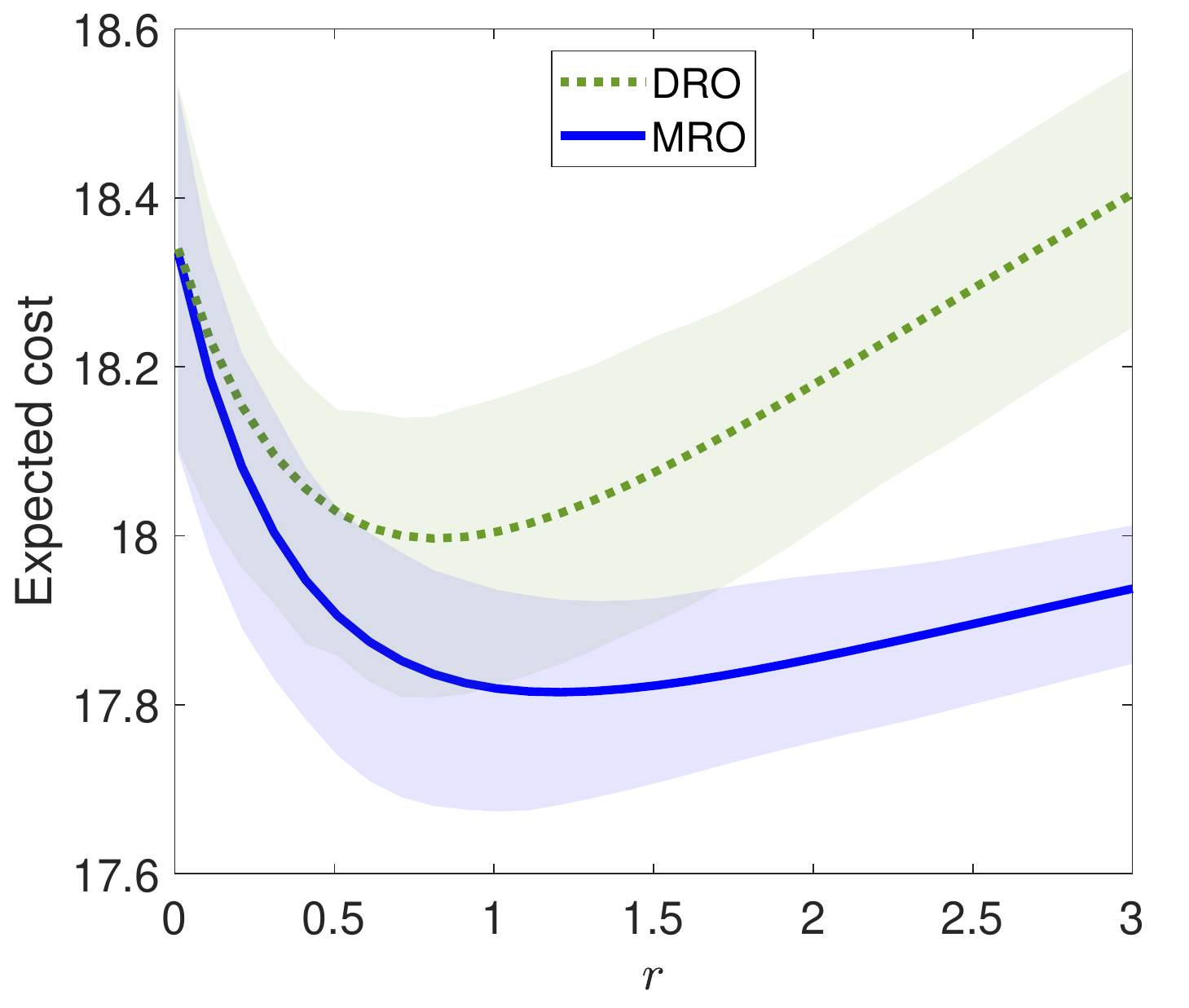}}\\ \vspace{10pt}
    \subfloat[$\mu=1$\label{fig:mu_nonzero}]{\includegraphics[width=0.9\columnwidth]{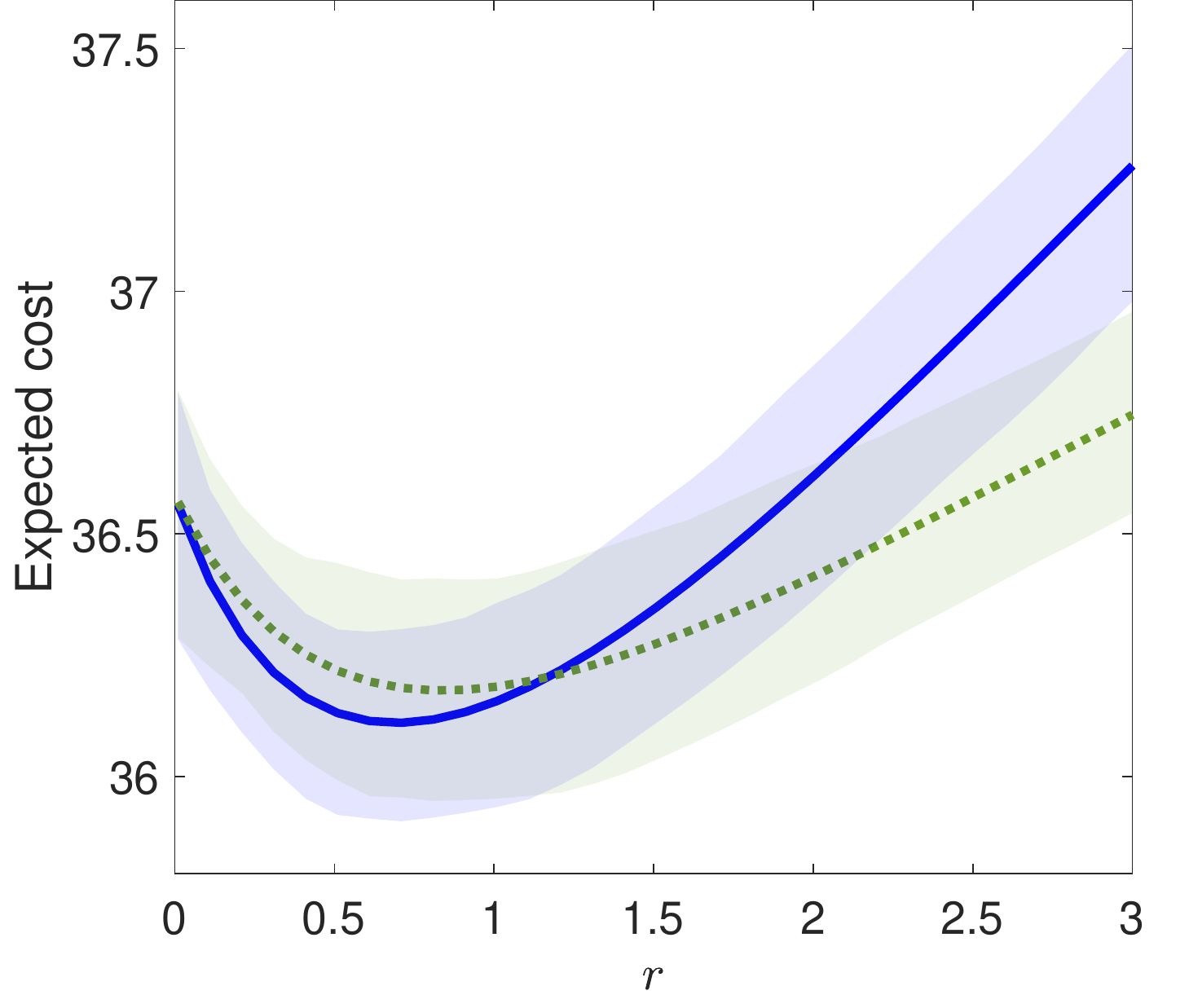}}
    \caption{These plots depict the expected cost (averaged over 100 trials) incurred by the MRO and DRO controllers under the true data-generating distribution $\mathcal{N}(\mu, \, I)$ versus the ambiguity set radius $r$. The shaded regions depict the range between the corresponding 20th and 80th empirical quantiles.
    }
    \label{fig:num_sim}
\end{figure}

We examine two different data-generating distributions:  one that has zero mean vector ($\mu =0$) and one that has a constant nonzero mean vector of all ones ($\mu =1$). 
For each distribution, we conduct 100 independent trials. In each trial, we draw $N = 50$ i.i.d. samples of the disturbance trajectory from the underlying distribution $\Ncal(\mu, \, I)$. 
We vary the radius of the ambiguity set  from  zero to three and compute the MRO controller for each radius value by solving the SDP \eqref{eq:sdp for MROC}. 
We compute the DRO controller by solving a SDP reformulation of \eqref{eq:DRO_original} that follows from Theorem \ref{thm:sdp of MROC}.

In Fig. \ref{fig:num_sim}, we plot the expected costs (averaged over 100 trials) incurred by the MRO and DRO controllers   as a function of the ambiguity set radius. In each case, the expectations are calculated with respect to the true  data-generating distribution.

When the radius $r$ of the ambiguity set is equal to zero, the MRO and DRO control problems become identical to the ambiguity-free stochastic control problem given by $\inf_{K \in \Kcal} \mathds{E}_{P_0}[J(Kw, \, w)]$, which can be interpreted as a \textit{certainty equivalent} approach to data-driven control.
For small values of $r$, we observe that both the MRO and DRO controllers improve upon the performance of the certainty equivalent controller.
In the zero-mean case (Fig. \ref{fig:mu_zero}), the MRO controller strictly outperforms the DRO controller for all radii in the specified range.
In the nonzero-mean case (Fig.~\ref{fig:mu_nonzero}), the MRO controller outperforms the DRO controller for small radius values, while the expected cost incurred by the DRO controller is smaller than that of the MRO controller for larger radii.
This reveals that, in some settings, the choice of ambiguity set radius $r$ can have an important effect on the superiority of one method over the other.
That said, for both distributions, the MRO controller achieves the smallest expected cost over the specified radius range, compared to the DRO controller.

\section{Conclusion} \label{sec:conclusion}

In this paper, we propose a distributionally robust approach to synthesizing minimax regret optimal controllers for   linear discrete-time systems subject to stochastic additive disturbances whose probability distribution is only known to lie in a type-2 Wasserstein ball of distributions. Building on existing strong duality results for Wasserstein distributionally robust optimization problems, we  provide an equivalent reformulation of the minimax regret optimal control problem as a tractable semidefinite program. As a direction of future research, it would be interesting to generalize the framework proposed in this paper to  enable the design of minimax regret optimal controllers with partial state feedback.

\bibliographystyle{IEEEtran}
\bibliography{references}{\markboth{References}{References}}

\end{document}